\documentclass[11pt]{article}%
\usepackage{amssymb,amsmath,amsfonts}
\usepackage{graphics}
\usepackage{graphicx}
\numberwithin{equation}{section}

\newenvironment{proof}{\noindent {\em {Proof}}.}{$\square$
\medskip}

\newtheorem{lemma}{Lemma}

\newtheorem{theorem}{Theorem}

\newtheorem{problem}{Problem}

\newcommand{\Aut}{{\rm Aut}}
\newcommand{\comment}[1]{}
\newif\ifpdf
\ifpdf \else \fi \textwidth = 6.5 in \textheight = 9 in
\begin{document}

\title{Rhomboidal $C_4C_8$ toris which are Cayley graphs}
\author{F.~Afshari\footnote {E-Mail:~fateme-afshari64@yahoo.com },~~ M. Maghasedi\footnote{Corresponding author}\footnote {E-Mail: maghasedi@kiau.ac.ir}\\
{\normalsize Department of Mathematics, Karaj Branch, Islamic Azad
University, Karaj, Iran}}
\date{}
\maketitle

\begin{abstract}
A $C_4C_8$ net is a trivalent decoration made by alternating
squares $C_4$ and octagons $C_8$. It can cover either a cylinder or a torus. In this paper
we determine rhomboidal $C_4C_8$ toris which are Cayley graphs.
\end{abstract}

\section{Introduction and preliminaries}
For notation and terminology not defined here we refer to
\cite{b}. Cayley \cite{c}, introduced the definition of Cayley
graph and over the time this graph has been investigated as an
interesting model of graphs by researchers. Let $G$ denote a
finite group and let $S$ be a subset of $G$ satisfying: (1)
$1\not \in S$, where $1$ denotes the identity of $G$, and (2)
$s\in S$ implies $s^{-1}\in S.$ The \textit{Cayley graph} $Cay(G,
S)$ on $G$ with connection set $S$ has the elements of $G$ as its
vertices and an edge joining $g$ to $sg$ for all $g \in G$ and
$s\in S$. Determining whether a graph is Cayley graph is one of
the most important problems. In particular, determining chemical Cayley graphs is very important. In 1994, Chung et al. showed that the
Buckyball, a soccer ball-like molecule, consisting of 60 carbon atoms is a Cayley graph on $A_5$, the alternating group on $5$ symbols
 \cite{CKS}, see also \cite[Example 7.16]{Chung}. Then in 2006, Meng et al. proved that
the only fulleren in which is a Cayley graph is the Buckyball \cite{MHZ}. Also it is shown in \cite{ad} that honeycomb toroidal graphs are Cayley
graphs. In this paper we consider the problem for rhomboidal $C_4C_8$ tori.

Carbon nanotubes form an interesting class of carbon
nanomaterials. These can be imagined as rolled sheets of graphite
about different axes. Iijima \cite{i}, in 1991 discovered carbon
nanotubes as multi-walled structures. Carbon nanotubes show
remarkable mechanical properties. A $C_4C_8$ net is a trivalent
decoration made by alternating squares $C_4$ and octagons $C_8$.
The rhomboidal $C_4C_8$ naotori $TRC_4C_8(R)[m,n]$ is constructed from squares and
octagons. Figure 1 in next page shows a $TRC_4C_8(R)[m,n]$ nanotori.

\begin{figure}[h]
\centering

\includegraphics[scale=0.5]{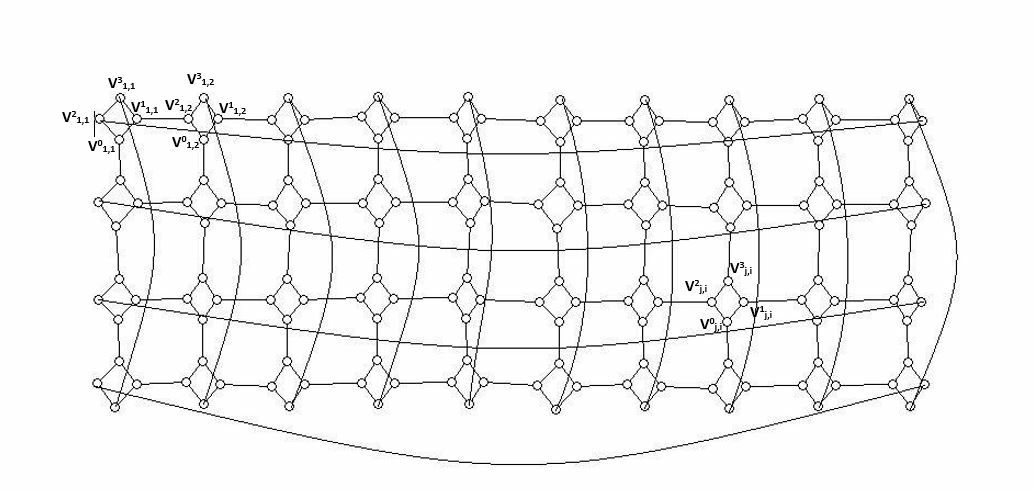}

\caption{$TRC_4C_8(R)[m,n]$ nanotori}

\end{figure}

Throughout this paper $\Gamma=TRC_4C_8(R)[m,n]$,
where $m$ is the number of rows and $n$ is the number of columns.
Thus $V(\Gamma)=\{v_{j,i}^{t}: 1\leq i\leq m, 1\leq j\leq n,
t=0,1,2,3\}$ and $$E(\Gamma)=\{v_{j,i}^{0}v_{j,i}^{1},
v_{j,i}^{0}v_{j,i}^{2},v_{j,i}^{0}v_{j+1,i}^{3},
v_{j,i}^{1}v_{j,i}^{0},v_{j,i}^{1}v_{j,i}^{3},$$
$$v_{j,i}^{1}v_{j,i+1}^{2},v_{j,i}^{2}v_{j,i}^{0},v_{j,i}^{2}v_{j,i}^{3},v_{j,i}^{2}v_{j,i-1}^{1},v_{j,i}^{3}v_{j,i}^{1},v_{j,i}^{3}v_{j,i}^{2},
v_{j,i}^{3}v_{j-1,i}^{0}:1\leq i\leq m, 1\leq j\leq n\},$$ where the sum/minus are in modulo
$m$ and the second sum/minus are in modulo $n$. Observe that
$\Gamma$ is a connected cubic graph of order $4mn$ and has $6mn$ edges.\

Our aim in this paper is to answer the following problem.

\begin{problem}
For which m,n, $TRC_4C_8(R)[m,n]$ is a Cayley graph?
\end{problem}

We use the following well-known theorem of Sabidussi.

\begin{theorem}[Sabidussi, \cite{s}]\label{Sabidussi}
Let $\Gamma$ be a graph. Then $\Gamma$ is a Cayley graph if and
only if $\Aut(\Gamma)$ contains a regular subgroup.
\end{theorem}

\section{Main results}

For $m\neq n $, we have some evidences that $TRC_4C_8(R)[m,n]$
may not be a Cayley graph. For for $m=3$ and $n=2,$ we computed
by GAP, that $\Aut(TRC_4C_8(R)[m,n])\cong D_{8}$ and has no
regular subgroup of order $24.$ In the rest of this section we
will prove that $TRC_4C_8(R)[m,n]$ is a Cayley graph whenever $
m=n.$

\begin{lemma}\label{lem1}
Let $ \Gamma=TRC_{4}C_{8}(R)[n,n] $ and $ g_{k}: v(\Gamma)
\mapsto v(\Gamma) $, $ 1 \leq k \leq 4 $,are the following maps:
\begin{align*}
 g_{1}&: v_{j,i}^{t} \mapsto v_{j,(i-1)}^{t}, t=0,1,2,3 \\
 g_{2}&: v_{j,i}^{t} \mapsto v_{(j+1),i}^{t} , t=0,1,2,3\\
 g_{3}&: v_{j,i}^{3} \mapsto v_{i,j}^{2}, ~ v_{j,i}^{2} \mapsto v_{i,j}^{3} ,~ v_{j,i}^{1} \mapsto v_{i,j}^{0},~v_{j,i}^{0} \mapsto v_{i,j}^{1}\\
 g_{4}&: v_{j,i}^{3} \mapsto v_{n-j+1,n-i+1}^{0}, ~ v_{j,i}^{2} \mapsto v_{n-j+1,n-i+1}^{1} ,~ v_{j,i}^{1} \mapsto v_{n-j+1,n-i+1}^{2},~v_{j,i}^{0} \mapsto v_{n-j+1,n-i+1}^{3}
\end{align*}
then $g_{k}\in\Aut(\Gamma)$, for all $k$.
\end{lemma}

\begin{proof}
We show that $g_{k}\in\Aut(\Gamma) $, $ 1\leq k \leq 4 $. It is
clear that $ g_{k} $ is a bijection for $ 1\leq k \leq 4$. Let
$v\in V(\Gamma)$ be an arbitrary vertex and $u\in N[v]$. We show
that $g_k(u)\in N[g_k(v)]$ for $1\leq k\leq 4$. Assume that
$v=v_{j,i}^t$, where $t\in\{0,1,2,3\}$, $j\in\{1,\cdots,n\}$ and
$i\in\{1,\cdots,n\}$. Note that
$N[v_{j,i}^3]=\{v_{j-1,i}^0,v_{j,i}^1,v_{j,i}^2\}$,
$N[v_{j,i}^2]=\{v_{j,i}^3,v_{j,i-1}^1,v_{j,i}^0\}$,
$N[v_{j,i}^1]=\{v_{j,i}^3,v_{j,i+1}^2,v_{j,i}^0\}$,
$N[v_{j,i}^0]=\{v_{j+1,i}^3,v_{j,i}^2,v_{j,i}^1\}$. Without loss
of generality, assume that $t=3$, and thus $v=v_{j,i}^3$. It
follows that $g_1(v)=v_{j,i-1}^3$, $g_2(v)=v_{j+1,i}^3$,
$g_3(v)=v_{i,j}^2$ and $g_4(v)=v_{n-j+1,n-i+1}^0$. Furthermore,
$u\in \{v_{j-1,i}^0,v_{j,i}^1,v_{j,i}^2\}$. Assume that
$u=v_{j-1,i}^0$. Then $g_1(u)=v_{j-1,i-1}^0$, $g_2(u)=v_{j,i}^0$,
$g_3(u)=v_{i,j-1}^1$ and $g_4(u)=v_{n-j+2,n-i+1}^3$. Thus for
$1\leq k\leq 4$, $g_k(u)\in N[g_k(v)]$, and so $g_k$ preserves
the adjacency. If $u\in \{v_{j,i}^1,v_{j,i}^2\}$, then similarly
we obtain that $g_k$ preserves the adjacency, for $1\leq k\leq
4$. Consequently, $g_k\in \Aut(\Gamma)$, for $1\leq k\leq 4$.
\end{proof}

\begin{lemma}\label{lem2}
Keeping the notations of Lemma \ref{lem1}, we have
\begin{align*}
&g_1^n=g_2^n=g_3^2=g_4^2=1,\\
&g_1g_2=g_2g_1,~ g_3g_4=g_4g_3, ~ g_1g_4=g_4g_1^{-1}, ~
g_2g_4=g_4g_2^{-1}, ~ g_2g_3=g_3g_1^{-1}.
\end{align*}
\end{lemma}
\begin{proof}
Let $v=v_{j,i}^t\in V(\Gamma) $, where $t\in\{0,1,2,3\}$, $j\in\{1,\cdots,n\}$ and
$i\in\{1,\cdots,n\}$. By definition of
$\Gamma=TUC_{4}C_{8}(R)[n,n]$, we have
$N[v_{j,i}^3]=\{v_{j,i}^2,v_{j,i}^1,v_{j-1,i}^0\}$,
$N[v_{j,i}^2]=\{v_{j,i}^3,v_{j,i-1}^1,v_{j,i}^0\}$,
$N[v_{j,i}^1]=\{v_{j,i}^3,v_{j,i+1}^2,v_{j,i}^0\}$,
$N[v_{j,i}^0]=\{v_{j+1,i}^3,v_{j,i}^2,v_{j,i}^1\}$. It is straightforward to see that $g_1^n(v)=g_2^n(v)=g_3^2(v)=g_4^2(v)=v$. We show that
$g_2g_3(v)=g_3g_1^{-1}(v) $. If $t=3$, then
 \begin{align*}
  &g_2g_3(v)=g_2g_3(v_{j,i}^3)=g_2(v_{j,i}^2)=v_{i+1,j}^2\\
 & g_3g_1^{-1}(v_{j,i}^3)=g_3(v_{j,i-(-1)}^3)=g_3(v_{j,i+1}^3)=v_{i+1,j}^2
\end{align*}
as desired. If $t=2$, then
\begin{align*}
  &g_2g_3(v)=g_2g_3(v_{j,i}^2)=g_2(v_{i,j}^3)=v_{i+1,j}^3\\
 & g_3g_1^{-1}(v_{j,i}^2)=g_3(v_{j,i-(-1)}^2)=g_3(v_{j,i+1}^2)=v_{i+1,j}^3
\end{align*} as desired. If $t=1$, then
\begin{align*}
  &g_2g_3(v)=g_2g_3(v_{j,i}^1)=g_2(v_{i,j}^0)=v_{i+1,j}^0\\
 & g_3g_1^{-1}(v_{j,i}^1)=g_3(v_{j,i-(-1)}^1)=g_3(v_{j,i+1}^1)=v_{i+1,j}^0
\end{align*} as desired. Finally assume that $t=0$. Then
\begin{align*}
  &g_2g_3(v_{j,i}^0)=g_2(v_{j,i}^1)=v_{i+1,j}^1\\
 & g_3g_1^{-1}(v_{j,i}^0)=g_3(v_{j,i-(-1)}^0)=g_3(v_{j,i+1}^0)=v_{i+1,j}^1.
\end{align*}
We next show that $g_2g_4(v)=g_4g_2^{-1}(v)$. If $t=3$, then
\begin{align*}
  &g_2g_4(v)=g_2g_4(v_{j,i}^3)=g_2(v_{n-j+1,n-i+1}^0)=v_{n-j+2,n-i+1}^0\\
 & g_4g_2^{-1}(v_{j,i}^3)=g_{4}(v_{(j-1),i}^3)=(v_{n-j+2,n-i+1}^0)
\end{align*} as desired. If $t=2$, then
\begin{align*}
  &g_2g_4(v_{j,i}^2)=g_2(v_{n-j+1,n-i+1}^1)=v_{n-j+2,n-i+1}^1\\
 & g_4g_2^{-1}(v_{j,i}^2)=g_4(v_{j-1,i}^2)=v_{n-j+2,n-i+1}^1
\end{align*}
If $t=1$, then
\begin{align*}
  &g_2g_4(v_{j,i}^1)=g_2(v_{n-j+1,n-i+1}^2)=v_{n-j+2,n-i+1}^0\\
 & g_4g_2^{-1}(v_{j,i}^1)=g_4(v_{j-1,i}^1)=v_{n-j+2,n-i+1}^2
\end{align*} Finally if $t=0$, then
\begin{align*}
  &g_2g_4(v_{j,i}^0)=g_2(v_{n-j+1,n-i+1}^3)=v_{n-j+2,n-i+1}^3\\
 & g_4g_2^{-1}(v_{j,i}^0)=g_4(v_{j-1,i}^0)=v_{n-j+2,n-i+1}^3.
\end{align*}
The proof for $g_1g_2(v)=g_2g_1(v)$, $g_3g_4(v)=g_4g_3(v)$, and $g_1g_4(v)=g_4g_1^{-1}(v)$ is similarly verified.
\end{proof}

Let us recall the definition of semidirect products of groups.
Let $G$ be a group, $H$ and $K$ be its subgroups, $G=HK$, $H$ be
normal in $G$, and $H\cap K=\{1\}$. Then we say $G$ is
\textit{semidirect product} of $H$ by $K$ and we write
$G=H\rtimes K$. It is a well-known fact in group theory that
every element $g$ of $G$ can be written uniquely as $g=hk$ for
some $h\in H$ and $k\in K$. If $K$ is also normal, then we say
$G$ is \textit{direct product} of $H$ and $K$ and we write
$G=H\times K$.

\begin{lemma}\label{lem3}
Keeping the notations of Lemma \ref{lem2}, let $
G=<g_1,g_2,g_3,g_4> $ be the group generated by $ g_{1} , g_{2},
g_{3} $ and $ g_{4}$. Then $G=H\rtimes K$, where $ H=<g_1,g_2>
\cong\Bbb C_n\times\Bbb C_n $ and $ K=<g_3,g_4> \cong\Bbb
C_2\times\Bbb C_2 $.
\end{lemma}
\begin{proof}
Put $ H=<g_1,g_2> $ and $ K=<g_3,g_4> $. Since $ g_1g_2=g_2g_1$,
we have $ H=<g_1>\times<g_2> \cong\Bbb C_n\times\Bbb C_n $ and
we have $ g_3g_4=g_4g_3 $ therefore $ K=<g_3>\times<g_4>
\cong\Bbb C_n\times\Bbb C_n $. Now we prove that $G=HK$, $H\lhd
G$ and $H\cap K=\{1\}$, which imply that  $G=H\rtimes K$.
Firstly, we prove that $ H\lhd G $. We know that, $
G=<g_1,g_2,g_3,g_4> $ and $ H=<g_1,g_2> $.
 $$ H\lhd G\Longleftrightarrow  g^{-1}Hg=H \Longleftrightarrow  g^{-1}g_{1}g\in H,~~g^{-1}g_{2}g\in H,~~\forall g\in G $$
It is obvious that, $ g_{1}^{-1}g_{1}g_{1}\in H $. We prove that
$ g_{2}^{-1}g_{1}g_{2}\in H $. By Lemma \ref{lem2},
$g_{1}g_{2}=g_{2}g_{1}$, and so $
g_{2}^{-1}g_{1}g_{2}=g_{2}^{-1}g_{2}g_{1}=g_{1}\in H $. Also we
show that, $ g_{3}^{-1}g_{1}g_{3}\in H.$
$$g_{2}g_{3}=g_{3}g^{-1}_{1} \Rightarrow g_{2}g_{3}g^{-1}_{3}
 =g_{3}g^{-1}_{1}g^{-1}_{3}\Rightarrow g_{2}
 =g_{3}g^{-1}_{1}g^{-1}_{3}\Rightarrow g^{-1}_{2}=g_{3}g_{1}g^{-1}_{3}.$$
Therefore, $ g^{-1}_{2}=g_{3}g_{1}g^{-1}_{3}=
g_{3}^{-1}g_{1}g_{3}\in H.$ According to Lemma \ref{lem2}, we
also have $g^{-1}_{4}g_{1}g_{4}\in H. $ In a similar way, we also
have, $$g^{-1}_{1}g_{2}g_{1}\in H,~ g^{-1}_{2}g_{2}g_{2}\in H,
 ~g^{-1}_{3}g_{2}g_{3}\in H,~g^{-1}_{4}g_{2}g_{4}\in H.$$
So for every generator of $G$ and $H$, like $g$ and $h$, respectively, we
have $ g^{-1}hg\in H,$ Therefore $ H\lhd G.$ We next
prove that $ H\cap K=1.$ Let $ H\cap K\neq1$. Then  $$ 1\neq x \in
H\cap K \Rightarrow x\in H, x\in K \Rightarrow
x=g^{m_{1}}_{1}g^{m_{2}}_{2},~x=g^{m_{3}}_{3}g^{m_{4}}_{4}.$$ If
we choose an arbitrary vertex $v=v_{1,1}^{1}$, then we have
$$ g^{m_{1}}_{1}g^{m_{2}}_{2}(v_{1,1}^{1})=g^{m_{3}}_{3}g^{m_{4}}_{4}(v_{1,1}^{1}).$$
But this is not possible, according to the maps $ g_{i},~1\leq i \leq 4 $. So $ H\cap K=1.$ Finally, we see that, $ G=HK$.
Every member of $G$ can be written in the form of $hk$, where
$h\in H$ and $k\in K$. As an example if we choose $g_{3}g_{1}\in
G$, we will see that $g_{3}g_{1}\in HK,$
$$ g_{3}g_{1}=g_{3}g_{1}g^{-1}_{3}g_{3}\in HK.$$
This complete the proof. Consequently, $G=H\rtimes K$
\end{proof}

Now we are ready to give the main theorem of this paper.

\begin{theorem} \label{cay}
Let $\Gamma=TRC_{4}C_{8}(R)[n,n]$. Then $\Gamma$ is a  Cayley
graph on $ G=<g_1,g_2,g_3,g_4>\cong(C_n\times C_n)\rtimes
(C_2\times C_2)$, where $g_k:V(\Gamma)\rightarrow V(\Gamma)$ are
maps
\begin{eqnarray*}
&g_1:&~~v_{j,i}^t\mapsto v_{j,(i-1)}^t, t=0,1,2,3\\
&g_2:&~~v_{j,i}^t\mapsto v_{(j+1),i}^t, t=0,1,2,3\\
&g_3:&~~ v_{j,i}^3\mapsto v_{i,j}^2,~ v_{j,i}^2\mapsto v_{i,j}^3,~
v_{j,i}^1\mapsto v_{i,j}^0,~
v_{j,i}^0\mapsto v_{i,j}^1\\
&g_4:&~~v_{j,i}^3\mapsto v_{n-j+1,n-i+1}^0,~ v_{j,i}^2\mapsto
v_{n-j+1,n-i+1}^1,~ v_{j,i}^1\mapsto v_{n-j+1,n-i+1}^2,~
v_{j,i}^0\mapsto v_{n-j+1,n-i+1}^3.
\end{eqnarray*}
\end{theorem}

\begin{proof}
According to Theorem \ref{Sabidussi}, it is sufficient to show
that $G$ is a regular subgroup of the automorphism group of
$\Gamma $. Thus we need to show that $G$ is a transitive subgroup
of $\Aut(\Gamma)$. According to Lemma \ref{lem1}, $ G\leqslant
\Aut(\Gamma) $. Now we show that $ G $ is transitive on $
V(\Gamma)$. Let $v_{j,i}^t$ and $v_{j',i'}^{t'}$ be two different
vertices of $V(\Gamma)$. Assume that $t=t'$. We will show that
there exists $g\in G$ such that $
g(v_{j,i}^{t})=v_{j^{'},i^{'}}^{t} $. Observe that $
G=<g_1,g_2,g_3,g_4> $. So $ g=
g^{r_{1}}_{1}g_{2}^{r_{2}}g_{3}^{r_{3}}g_{4}^{r_{4}}$. Consider
the maps $g_{k}$ that $1\leq k\leq 4$ defined in the Lemma
\ref{lem1}. Notice that $g_{3}$ and $ g_{4}$ realize that $
r_{3}=0$ and $ r_{4}=0$. Therefore, $
g^{r_{1}}_{1}g_{2}^{r_{2}}(v_{j,i}^{t})=v_{j^{'},i^{'}}^{t} $. So
the following computation result in $ g=g_1^{i-i'}g_2^{j'-j} $.
$$ g_{1}^{r_{1}}g_{2}^{r_{2}}(v_{j,i}^{t})= g_{1}^{r_{1}}(v_{j+r_{2},i}^{t})=
v_{j+r_{2},i-r_{1}}^{t}=v_{j^{'},i^{'}}^{t}.$$ So $
r_{2}=j^{'}-j $ and $ r_{1}=i-i^{'}$. Therefore, $
g_1^{i-i'}g_2^{j'-j}(v_{j,i}^t)=v_{j',i'}^t $. Now let $ t\neq t'
$. We prove that,
$$\forall v_{j,i}^t ,v^{t^{'}}_{j',i'} \in V(\Gamma)
~\exists g\in
\Aut(\Gamma)~s.t.~g(v_{j,i}^t)=v^{t^{'}}_{j^{'},i^{'}}.$$
For this purpose, we will prove following equalities:
\begin{align*}
&g_1^{j-i'}g_2^{j'-i}g_3(v_{j,i}^3)=v_{j',i'}^2,\\
&g_1^{-i'-i+1}g_2^{j'+j-1}g_4(v_{j,i}^3)=v_{j',i'}^0,\\
&g_1^{-i'-j+1}g_2^{j'+i-1}g_3g_4(v_{j,i}^3)=v_{j',i'}^1,\\
&g_1^{-i'-i+1}g_2^{j'+j-1}g_4(v_{j,i}^2)=v_{j',i'}^1\\
&g_1^{-i'-j+1}g_2^{j'+i-1}g_4g_3(v_{j,i}^2)=v_{j',i'}^0,\\
&g_1^{j-i'}g_2^{j'-i}g_3(v_{j,i}^1)=v_{j',i'}^0.
\end{align*}
We prove only the first equality, the others can be proved
similarly. To prove that $
g_1^{j-i'}g_2^{j'-i}g_3(v_{j,i}^3)=v_{j',i'}^2,$ note that $
G=<g_1,g_2,g_3,g_4> $. So $ g=
g^{r_{1}}_{1}g_{2}^{r_{2}}g_{3}^{r_{3}}g_{4}^{r_{4}}$. Observe
that the maps $ g_3,g_4 $ realize that $ r_{4}=0 $ and $r_{3}=1$.
So $ g_1^{r_{1}}g_2^{r_{2}}g_3(v^{3}_{j,i})=v_{j',i'}^2$. We
should find $ r_{1}$ and $r_{2}$. It follows that
$$g_1^{r_{1}}g_2^{r_{2}}g_{3}(v^{3}_{j,i})
=g^{r_{1}}_{1}g^{r_{2}}_{2}(v^{2}_{i,j})=g^{r_{1}}_{1}(v_{(i+r_{2}),j}^{2})
=v_{(i+r_{2}),(j-r_{1})}^{2}=v_{j',i'}^2.$$ Therefore, $
i+r_{2}=j^{'}$ and $ j-r_{1}=i^{'}$, and so $ r_{2}=j^{'}-i$ and $
r_{1}=j-i^{'}$. As a result, $ g=g_1^{j-i'}g_2^{j'-i}g_3$.
Thus for every $v_{j,i}^t,v_{j',i'}^{t'}\in V(\Gamma)$
there exists $g\in G$ such that
$g(v_{j,i}^t)=v_{j',i'}^{t'}$, and this means that $G$
is transitive. By Lemma \ref{lem3}, we have  $G=H\rtimes K$, where
$ H=<g_1>\times<g_2> \cong\Bbb C_n\times\Bbb C_n $
 and  $ K=<g_3>\times<g_4> \cong\Bbb C_2\times\Bbb C_2 $. Therefore, $G\cong (C_n\times
C_n)\rtimes (C_2\times C_2)$ and $|G|=4n^2$.
\end{proof}

\end{document}